\documentclass{amsart}

\usepackage{amssymb}
\usepackage{amsmath}
\usepackage{amsthm}
\usepackage{amsfonts}

\usepackage{a4wide}
\usepackage{enumerate}

\newtheorem{theorem}{Theorem}[section]

\newtheorem{corollary}{Corollary}[section]

\theoremstyle{definition}
\newtheorem{definition}{Definition}[section]
\newtheorem{remark}{Remark}[section]

\numberwithin{equation}{section}

\begin{document}


\title{Periodic sequences of \(p\)-class tower groups}

\author{Daniel C. Mayer}
\address{Naglergasse 53\\8010 Graz\\Austria}
\email{algebraic.number.theory@algebra.at}
\urladdr{http://www.algebra.at}

\thanks{Research supported by the Austrian Science Fund (FWF): P 26008-N25}


\date{March 31, 2015}

\dedicatory{Dedicated to the memory of Emil Artin}

\begin{abstract}
Recent examples of
periodic bifurcations in descendant trees of finite \(p\)-groups
with \(p\in\lbrace 2,3\rbrace\)
are used to show that the possible \(p\)-class tower groups \(G\)
of certain multiquadratic fields \(K\) with
\(p\)-class group of type \((2,2,2)\), resp. \((3,3)\),
form periodic sequences in the descendant tree of the
elementary abelian root \(C_2^3\), resp. \(C_3^2\).
The particular vertex of the periodic sequence
which occurs as the \(p\)-class tower group \(G\) of an assigned field \(K\)
is determined uniquely by the \(p\)-class number of
a quadratic, resp. cubic, auxiliary field \(k\), associated unambiguously to \(K\).
Consequently, the hard problem of identifying the \(p\)-class tower group \(G\)
is reduced to an easy computation of low degree arithmetical invariants.
\end{abstract}

\maketitle



\section{Introduction}
\label{s:Intro}

In this article, we establish class field theoretic applications
of the purely group theoretic discovery of
periodic bifurcations in descendant trees of finite \(p\)-groups,
as described in our previous papers
\cite[\S\S\ 21--22, pp.182--193]{Ma1}
and
\cite[\S\ 6.2.2]{Ma2},
and summarized in section \S\
\ref{s:PeriodicBifurcations}.

The infinite families of Galois groups of \(p\)-class field towers
with \(p\in\lbrace 2,3\rbrace\)
which are presented in sections \S\S\
\ref{s:TwoStage}
and
\ref{s:ThreeStage}
can be divided into different kinds.
Either they form infinite periodic sequences of uniform step size \(1\),
and thus of fixed coclass.
These are the classical and well-known \textit{coclass sequences}
whose virtual periodicity has been proved independently
by du Sautoy and by Eick and Leedham-Green (see
\cite[\S\ 7, pp.167--168]{Ma1}).
Or they arise from infinite paths of directed edges in descendant trees
whose vertices reveal periodic bifurcations (see
\cite[Thm.21.1, p.182]{Ma1},
\cite[Thm.21.3, p.185]{Ma1},
and
\cite[Thm.6.4]{Ma2}).
Extensive finite parts of the latter have been constructed computationally
with the aid of the \(p\)-group generation algorithm
by Newman
\cite{Nm}
and O'Brien
\cite{Ob}
(see
\cite[\S\S 12--18]{Ma1}),
but their indefinitely repeating periodic pattern
has not been proven rigorously, so far.
They can be of uniform step size \(2\), as in \S\
\ref{s:TwoStage},
or of mixed alternating step sizes \(1\) and \(2\), as in \S\
\ref{s:ThreeStage},
whence their coclass increases unboundedly.



\section{Periodic bifurcations in trees of \(p\)-groups}
\label{s:PeriodicBifurcations}

For the specification of finite \(p\)-groups
throughout this article,
we use the identifiers of the SmallGroups database
\cite{BEO1,BEO2}
and of the ANUPQ-package
\cite{GNO}
implemented in the computational algebra systems GAP
\cite{GAP}
and MAGMA
\cite{BCP,BCFS,MAGMA},
as discussed in
\cite[\S\ 9, pp.168--169]{Ma1}.

The first periodic bifurcations were discovered in August 2012
for the descendant trees of the \(3\)-groups
\(Q=\langle 729,49\rangle\) and \(U=\langle 729,54\rangle\) (see
\cite[\S\ 3, p.163]{Ma1}
and
\cite[Thm.21.3, p.185]{Ma1}),
having abelian quotient invariants \((3,3)\),
when we, in collaboration with Bush, conducted a search for Schur \(\sigma\)-groups
as possible candidates for Galois groups
\(\mathrm{G}_3^{\infty}(K)=\mathrm{Gal}(\mathrm{F}_3^{\infty}(K)\vert K)\)
of three-stage towers of \(3\)-class fields
over complex quadratic base fields \(K=\mathbb{Q}(\sqrt{d})\)
with \(d\le -9748\) and a certain type of \(3\)-principalization
\cite[Cor.4.1.1, p.775]{BuMa}.
The result in
\cite[Thm.4.1, p.774]{BuMa}
will be generalized to more principalization types
and groups of higher nilpotency class in section \S\
\ref{s:ThreeStage}.

\noindent
Similar phenomena were found in May 2013 for the trees with roots
\(\langle 2187,168\rangle\) and \(\langle 2187,181\vert 191\rangle\)
of type \((9,3)\) but have not been published yet,
since we first have to present a classification of
all metabelian \(3\)-groups with abelianization \((9,3)\).

At the beginning of 2014, we investigated the root \(\langle 729,45\rangle\),
which possesses an infinite balanced cover
\cite[Dfn.6.1]{Ma2},
and found periodic bifurcations in its decendant tree
\cite[Thm.6.4]{Ma2}).

In January 2015, we studied complex bicyclic biquadratic fields
\(K=\mathbb{Q}(\sqrt{-1},\sqrt{d})\),
called \textit{special Dirichlet fields} by Hilbert
\cite{Hi},
for whose \(2\)-class tower groups \(\mathrm{G}_2^{\infty}(K)\)
presentations had been given by Azizi, Zekhnini and Taous
\cite[Thm.2(4)]{AZT},
provided the radicand \(d\) exhibits a certain prime factorization
which ensures a \(2\)-class group \(\mathrm{Cl}_2(K)\) of type \((2,2,2)\).

In section \S\
\ref{s:TwoStage},
we use the viewpoint of descendant trees of finite metabelian \(2\)-groups
and our discovery of periodic bifurcations in the tree with root \(\langle 32,34\rangle\)
\cite[Thm.21.1, p.182]{Ma1}
to prove a group theoretic restatement
of the main result in the paper
\cite{AZT},
which connects pairs \((m,n)\) of positive integer parameters
with vertices of the descendant tree \(\mathcal{T}(\langle 8,5\rangle)\)
by means of an injective mapping
\((m,n)\mapsto G_{m,n}\),
as shown impressively in Figure
\ref{fig:IdAndParam222}.



\section{Pattern recognition via Artin transfers}
\label{s:ArtinPattern}

Let \(p\) denote a prime number and
suppose that \(G\) is a finite \(p\)-group
or an infinite pro-\(p\) group with finite abelianization \(G/G^\prime\)
of order \(p^v\) with a positive integer exponent \(v\ge 1\).

In this situation, there exist \(v+1\) layers
\[\mathrm{Lyr}_n(G):=\lbrace G^\prime\le H\le G\mid (G:H)=p^n\rbrace,\text{ for }0\le n\le v,\]
of intermediate normal subgroups \(H\unlhd G\)
between \(G\) and its commutator subgroup \(G^\prime\).
For each of them, we denote by
\(T_{G,H}:G\to H/H^\prime\)
the Artin transfer homomorphism from \(G\) to \(H\)
\cite{Ar}.
In our recent papers
\cite[\S\ 3]{Ma2}
and
\cite{Ma},
the components of
the multiple-layered \textit{transfer target type} (TTT)
\(\tau(G)=\lbrack\tau_0(G);\ldots;\tau_v(G)\rbrack\) of \(G\),
resp. the multiple-layered \textit{transfer kernel type} (TKT)
\(\varkappa(G)=\lbrack\varkappa_0(G);\ldots;\varkappa_v(G)\rbrack\) of \(G\),
were defined by
\[\tau_n(G):=(H/H^\prime)_{H\in\mathrm{Lyr}_n(G)},
\text{ resp. }
\varkappa_n(G):=(\ker(T_{G,H}))_{H\in\mathrm{Lyr}_n(G)},
\text{ for }0\le n\le v.\]

The following information is known
\cite{Ma}
to be crucial for identifying
the metabelianization \(G/G^{\prime\prime}\) of a \(p\)-class tower group \(G\),
but usually does not suffice to determine \(G\) itself.

\begin{definition}
\label{dfn:ArtinPattern}
By the \textit{Artin pattern} of \(G\) we understand the pair

\begin{equation}
\label{eqn:ArtinPattern}
\mathrm{AP}(G):=(\tau(G);\varkappa(G))
\end{equation}

\noindent
consisting of the multiple-layered TTT \(\tau(G)\)
and the multiple-layered TKT \(\varkappa(G)\)
of \(G\).

\noindent
If \(G\) is the \(p\)-tower group of a number field \(K\),
then we put \(\mathrm{AP}(K):=\mathrm{AP}(G)\)
and speak about the \textit{Artin pattern} of \(K\).
\end{definition}

As Emil Artin
\cite{Ar}
pointed out in 1929 already,
using a more classical terminology,
the concepts transfer target type (TTT) and transfer kernel type (TKT)
of a base field \(K\),
which we have now combined to the Artin pattern \((\tau(K);\varkappa(K))\) of \(K\),
require a \textit{non-abelian} setting of unramified extensions of \(K\).
The reason is that the derived subgroup \(H^\prime\) of an intermediate group \(G^\prime<H<G\)
between the \(p\)-tower group \(G\) of \(K\) and its commutator subgroup \(G^\prime\)
is an intermediate group between \(G^\prime\) and the second derived subgroup \(G^{\prime\prime}\).
Therefore, the TTT \(\tau(G)\) of the \(p\)-tower group \(G=G_p^\infty(K)\) coincides with the
TTT \(\tau(G_p^n(K))\) of any higher derived quotient \(G_p^n(K)\simeq G/G^{(n)}\),
for \(n\ge 2\) but not for \(n=1\),
since \(H/H^\prime\simeq(H/G^{(n)})/(H^\prime/G^{(n)})\),
according to the isomorphism theorem.
Similarly, we have the coincidence of TKTs \(\varkappa(G_p^n(K))=\varkappa(G)\), for \(n\ge 2\).



\section{Two-stage towers of \(2\)-class fields}
\label{s:TwoStage}

As our first application of periodic bifurcations in trees of \(2\)-groups,
we present a family of biquadratic number fields \(K\)
with \(2\)-class group \(\mathrm{Cl}_2(K)\) of type \((2,2,2)\),
discovered by Azizi, Zekhnini and Taous
\cite{AZT},
whose \(2\)-class tower groups \(G=\mathrm{G}_2^{\infty}(K)\) are
conjecturally distributed over infinitely many periodic coclass sequences, without gaps.

This claim is stronger than the statements in the following Theorem
\ref{thm:TwoStage}.
The proof firstly consists of a group theoretic construction
of all possible candidates for \(G\), identified by their Artin pattern,
up to nilpotency class \(\mathrm{cl}(G)\le 12\)
and coclass \(\mathrm{cc}(G)\le 13\),
thus having a maximal logarithmic order
\(\log_2(\mathrm{ord}(G))\le 25\).
(The first part is independent of the actual realization
of the possible groups \(G\) as \(2\)-tower groups of suitable fields \(K\).)
Secondly, evidence is provided of the realization of
at least all those groups constructed in the first part
whose logarithmic order does not exceed \(11\).
The second part (see \S\
\ref{s:CompRslt})
is done by computing the Artin pattern
of sufficiently many fields \(K\)
or by using more sophisticated ideas, presented in Theorem
\ref{thm:TwoStage}.

\begin{remark}
\label{rmk:TwoStage}

Generally, the first layer of the transfer kernel type \(\varkappa_1(G)\) of \(G\)
will turn out to be a permutation
\cite[Dfn.21.1, p.182]{Ma1}
of the seven planes
in the \(3\)-dimensional \(\mathbb{F}_2\)-vector space
\(G/G^\prime\simeq\mathrm{Cl}_2(K)\).
We are going to use the notation of
\cite[Thm.21.1 and Cor.21.1]{Ma1}.

\end{remark}

\begin{theorem}
\label{thm:TwoStage}

Let \(K=\mathbb{Q}(\sqrt{-1},\sqrt{d})\)
be a complex bicyclic biquadratic Dirichlet field with radicand \(d=p_1p_2q\),
where \(p_1\equiv 1\pmod{8}\), \(p_2\equiv 5\pmod{8}\) and
\(q\equiv 3\pmod{4}\) are prime numbers such that
\(\left(\frac{p_1}{p_2}\right)=-1\) and \(\left(\frac{p_1}{q}\right)=-1\).

\noindent
Then the \(2\)-class group \(\mathrm{Cl}_2(K)\) of \(K\) is of type \((2,2,2)\),
the \(2\)-class field tower of \(K\) is metabelian (with exactly two stages),
and the isomorphism type of the Galois group
\(G=\mathrm{G}_2^{\infty}(K)=\mathrm{Gal}(\mathrm{F}_2^{\infty}(K)\vert K)\)
of the maximal unramified pro-\(2\) extension \(\mathrm{F}_2^{\infty}(K)\) of \(K\)
is characterized uniquely by the pair of positive integer parameters \((m,n)\)
defined by the \(2\)-class numbers
\(h_2(k_1)=2^{m+1}\) and  \(h_2(k_2)=2^{n}\)
of the complex quadratic fields
\(k_1=\mathbb{Q}(\sqrt{-p_1})\) and
\(k_2=\mathbb{Q}(\sqrt{-p_2q})\).

\noindent
The Legendre symbol \(\left(\frac{p_2}{q}\right)\) decides whether \(G\) is a descendant of
\(\langle 32,34\rangle\) or \(\langle 32,35\rangle\):

\begin{itemize}

\item
\(\left(\frac{p_2}{q}\right)=-1\)
\(\Longleftrightarrow\) \((m\ge)n=1\)
\(\Longleftrightarrow\) the first layer TKT \(\varkappa_1(G)\)
is a permutation with five fixed points and a single \(2\)-cycle
\(\Longleftrightarrow\) \(G\) belongs to the mainline

\begin{equation}
\label{eqn:MainLine35}
M_{0,k}:=\langle 32,35\rangle(-\#1;1)^k,\text{ with }k=m-1\ge 0,
\end{equation}

\noindent
of the coclass tree
\(\mathcal{T}^3(\langle 32,35\rangle)\).

\item
\(\left(\frac{p_2}{q}\right)=+1\)
\(\Longleftrightarrow\) \(n>1\)
\(\Longleftrightarrow\) the first layer TKT \(\varkappa_1(G)\)
is a permutation with a single fixed point and three \(2\)-cycles
\(\Longleftrightarrow\) \(G\) is a descendant of the group \(\langle 32,34\rangle\),
that is \(G\in\mathcal{T}(\langle 32,34\rangle)\).

\end{itemize}

\noindent
More precisely, in the second case
the following equivalences hold in dependence on the parameters \(m,n\le\ell\),
where \(\ell\le 11\) denotes a foregiven upper bound:

\begin{itemize}

\item
\(m\ge n\ge 2\) (with \(n\) fixed)
\(\Longleftrightarrow\) \(G\) belongs to the mainline

\begin{equation}
\label{eqn:MainLine}
M_{j+1,k}:=\langle 32,34\rangle(-\#2;1)^j-\#2;2(-\#1;1)^k,\text{ with fixed }j=n-2
\end{equation}

\noindent
and varying \(k=m-n\ge 0\),
of the coclass tree
\(\mathcal{T}^{n+2}(\langle 32,34\rangle(-\#2;1)^{n-2}-\#2;2)\).

\item
\(n>m\ge 1\) (with \(m\) fixed)
\(\Longleftrightarrow\) \(G\) belongs to the unique periodic coclass sequence

\begin{equation}
\label{eqn:Sequence}
V_{j,k}:=\langle 32,34\rangle(-\#2;1)^j(-\#1;1)^k-\#1;2,\text{ with fixed }j=m-1
\end{equation}

\noindent
and varying \(k=n-m-1\ge 0\),
whose members possess a permutation as their first layer transfer kernel type,
of the coclass tree
\(\mathcal{T}^{m+2}(\langle 32,34\rangle(-\#2;1)^{m-1})\).

\end{itemize}

\end{theorem}

We add a corollary which gives the Artin pattern of the groups in Theorem
\ref{thm:TwoStage},
firstly, since it is interesting in its own right, and
secondly, because we are going to use its proof
as a starting point for the proof of Theorem
\ref{thm:TwoStage}.



\begin{corollary}
\label{cor:TwoStage}
Under the assumptions of Theorem
\ref{thm:TwoStage},
the Artin pattern \(\mathrm{AP}(G)=(\tau(G);\varkappa(G))\)
of the \(2\)-tower group \(G=\mathrm{G}_2^\infty(K)\)
of the biquadratic field \(K=\mathbb{Q}(\sqrt{-1},\sqrt{d})\) is given as follows:

The ordered multi-layered transfer target type (TTT)
\(\tau(G)=\lbrack\tau_0;\tau_1;\tau_2;\tau_3\rbrack\)
of the Galois group \(G\) is given by \(\tau_0=(1^3)\), \(\tau_3=(m,n)\), and

\begin{equation}
\label{eqn:1stLayerTTT}
\tau_1=
\begin{cases}
\lbrack (m+1,2),(2,1)^2,(1^3)^2,(2,1)^2\rbrack,\text{ if }\left(\frac{p_2}{q}\right)=-1, \\
\lbrack (m+1,n+1),(1^3)^6\rbrack,\text{ else},
\end{cases}
\end{equation}

\begin{equation}
\label{eqn:2ndLayerTTT}
\tau_2=
\begin{cases}
\lbrack (m+1,1),(m,2),(m+1,1),(2,1)^4\rbrack,\text{ if }\left(\frac{p_2}{q}\right)=-1, \\
\lbrack (m+1,n),(m,n+1),(\max(m+1,n+1),\min(m,n)),(1^3)^4\rbrack,\text{ else}.
\end{cases}
\end{equation}

\noindent
If we now denote by \(N_i:=\mathrm{Norm}_{K_i\vert K}(\mathrm{Cl}_2(K_i))\), \(1\le i\le 7\),
the norm class groups of the seven unramified quadratic extensions \(K_i\vert K\),
then the ordered multi-layered transfer kernel type (TKT)
\(\varkappa(G)=\lbrack\varkappa_0;\varkappa_1;\varkappa_2;\varkappa_3\rbrack\)
of the Galois group \(G\) is given by \(\varkappa_0=1\), \(\varkappa_2=(0^7)\), \(\varkappa_3=(0)\), and

\begin{equation}
\label{eqn:1stLayerTKT}
\varkappa_1=
\begin{cases}
(N_1,N_2,N_3,N_5,N_4,N_6,N_7),\text{ if }\left(\frac{p_2}{q}\right)=-1, \\
(N_1,N_3,N_2,N_5,N_4,N_7,N_6),\text{ else}.
\end{cases}
\end{equation}

\noindent
Thus, \(\varkappa_1\) is always a permutation of the norm class groups \(N_i\).
For \(\left(\frac{p_2}{q}\right)=-1\) it contains five fixed points and a single \(2\)-cycle,
and otherwise it contains a single fixed point and three \(2\)-cycles.

\end{corollary}



\begin{proof}
The underlying order of the \(7\) unramified quadratic, resp. bicyclic biquadratic,
extensions of \(K\) is taken from
\cite[\S\ 2.1, Thm.1,(3),(5)]{AZT}.

For the TTT we use logarithmic abelian type invariants as explained in
\cite[\S\ 2]{Ma2}.
\(\tau_0\) is taken from
\cite[\S\ 2.2, Thm.2,(1)]{AZT},
\(\tau_1,\tau_2\) from
\cite[\S\ 2.3, Thm.3,(1),(2)]{AZT},
and \(\tau_3\) from
\cite[\S\ 2.2, Thm.2,(5)]{AZT}.

Concerning the TKT, \(\varkappa_0\) is trivial,
\(\varkappa_1,\varkappa_2\) are taken from
\cite[\S\ 2.3, Thm.3,(3)--(5)]{AZT},
and \(\varkappa_3\) is total,
due to the Hilbert/Artin/Furtw\"angler principal ideal theorem.
\end{proof}



\begin{proof}
(Proof of Theorem
\ref{thm:TwoStage})

Firstly, the equivalence
\(\left(\frac{p_2}{q}\right)=-1\) \(\Longleftrightarrow\) \(n=1\)
is proved in
\cite[\S\ 3, Lem.5]{AZT}.

\noindent
Next, we use the Artin pattern of \(G\), as given in Corollary
\ref{cor:TwoStage},
to narrow down the possibilities for \(G\).
The possible class-\(2\) quotients of \(G\)
are exactly the immediate descendants of the root \(\langle 8,5\rangle\),
that is,
three vertices \(\langle 16,11\ldots 13\rangle\) of step size \(1\),
nine vertices \(\langle 32,27\ldots 35\rangle\) of step size \(2\),
and ten vertices \(\langle 64,73\ldots 82\rangle\) of step size \(3\).
Among all descendants of \(\langle 8,5\rangle\),
the mainline vertices of the tree \(\mathcal{T}(\langle 32,35\rangle)\)
are characterized uniquely by the fact that
their first layer TKT \(\varkappa_1\)
is a permutation with five fixed points and a single \(2\)-cycle,
and that their first layer TTT \(\tau_1\)
contains the unique polarized (i.e. parameter dependent) component \((m+1,2)\).
Note that
the mainline vertices of the tree \(\mathcal{T}(\langle 32,31\rangle)\)
reveal the same six stable (i.e. parameter independent) components
\(((1^3)^2,(2,1)^4)\) of the accumulated (unordered) first layer TTT \(\tau_1\),
but their first layer TKT \(\varkappa_1\) contains three \(2\)-cycles,
similarly as for descendants of \(\langle 32,34\rangle\).
However, vertices of the complete descendant tree \(\mathcal{T}(\langle 32,34\rangle)\)
are characterized uniquely by six stable components
\(((1^3)^6)\) of their first layer TTT \(\tau_1\).

\noindent
So far, we have been able to single out
that \(G\) must be a descendant of
either \(\langle 32,34\rangle\) or \(\langle 32,35\rangle\),
by means of Artin patterns,
without knowing a presentation.
Now, the parametrized presentation for the group \(G=G_{m,n}\) in
\cite[\S\ 2.2, Thm.2,(4)]{AZT},

\begin{equation}
\label{eqn:Presentation}
G_{m,n} = \langle\rho,\sigma,\tau\mid\rho^4=\sigma^{2^{n+1}}=\tau^{2^{m+1}}=1,\rho^2=\sigma^{2^n},
\lbrack\rho,\sigma\rbrack=\sigma^2,\lbrack\rho,\tau\rbrack=\tau^2,\lbrack\sigma,\tau\rbrack=1\rangle,
\end{equation}

\noindent
is used as input for a Magma program script
\cite{BCFS,MAGMA}
which identifies a \(2\)-group, given by generators and relations,

\begin{center}
\texttt{Group}\(<\rho,\sigma,\tau\mid\text{ relator words in }\rho,\sigma,\tau>\),
\end{center}

\noindent
with the aid of the following functions:

\begin{itemize}

\item
\texttt{CanIdentifyGroup()} and \texttt{IdentifyGroup()} if \(\lvert G\rvert\le 2^8\),

\item
\texttt{IsInSmallGroupDatabase(), pQuotient(), NumberOfSmallGroups(), SmallGroup()} \\
and \texttt{IsIsomorphic()} if \(\lvert G\rvert=2^9\), and

\item
\texttt{GeneratepGroups()}, resp. a recursive call of \texttt{Descendants()} \\
(using \texttt{NuclearRank()} for the recursion),
and \texttt{IsIsomorphic()} if \(\lvert G\rvert\ge 2^{10}\).
\end{itemize}

The output of the Magma script is in perfect accordance with the
pruned descendant tree \(\mathcal{T}_\ast(\langle 8,5\rangle)\),
as described in Theorem 21.1 and Corollary 21.1 of
\cite[pp.182--183]{Ma1}.

Finally, the class and coclass of \(G\) are given in
\cite[\S\ 2.2, Thm.2,(6)]{AZT}.
\end{proof}



{\tiny

\begin{figure}[hb]
\caption{Pairs \((m,n)\) of parameters distributed over \(\mathcal{T}_\ast(\langle 8,5\rangle)\)}
\label{fig:IdAndParam222}

\input{IdAndParam222}

\end{figure}

}



{\tiny

\begin{figure}[hb]
\caption{Minimal radicands \(d\) distributed over \(\mathcal{T}_\ast(\langle 8,5\rangle)\)}
\label{fig:MinRad222}


\setlength{\unitlength}{0.8cm}
\begin{picture}(18,22)(-4,-21)

\put(-5,0.5){\makebox(0,0)[cb]{Order}}
\put(-5,0){\line(0,-1){18}}
\multiput(-5.1,0)(0,-2){10}{\line(1,0){0.2}}
\put(-5.2,0){\makebox(0,0)[rc]{\(8\)}}
\put(-4.8,0){\makebox(0,0)[lc]{\(2^3\)}}
\put(-5.2,-2){\makebox(0,0)[rc]{\(16\)}}
\put(-4.8,-2){\makebox(0,0)[lc]{\(2^4\)}}
\put(-5.2,-4){\makebox(0,0)[rc]{\(32\)}}
\put(-4.8,-4){\makebox(0,0)[lc]{\(2^5\)}}
\put(-5.2,-6){\makebox(0,0)[rc]{\(64\)}}
\put(-4.8,-6){\makebox(0,0)[lc]{\(2^6\)}}
\put(-5.2,-8){\makebox(0,0)[rc]{\(128\)}}
\put(-4.8,-8){\makebox(0,0)[lc]{\(2^7\)}}
\put(-5.2,-10){\makebox(0,0)[rc]{\(256\)}}
\put(-4.8,-10){\makebox(0,0)[lc]{\(2^8\)}}
\put(-5.2,-12){\makebox(0,0)[rc]{\(512\)}}
\put(-4.8,-12){\makebox(0,0)[lc]{\(2^9\)}}
\put(-5.2,-14){\makebox(0,0)[rc]{\(1\,024\)}}
\put(-4.8,-14){\makebox(0,0)[lc]{\(2^{10}\)}}
\put(-5.2,-16){\makebox(0,0)[rc]{\(2\,048\)}}
\put(-4.8,-16){\makebox(0,0)[lc]{\(2^{11}\)}}
\put(-5.2,-18){\makebox(0,0)[rc]{\(4\,096\)}}
\put(-4.8,-18){\makebox(0,0)[lc]{\(2^{12}\)}}
\put(-5,-18){\vector(0,-1){2}}

\put(-3,0){\makebox(0,0)[cc]{\(\mathcal{G}(2,2)\)}}

\put(-4.05,-0.05){\framebox(0.1,0.1){}}
\put(-3.9,0.2){\makebox(0,0)[lb]{\(\langle 5\rangle\)}}
\put(-4,0){\line(1,-1){4}}
\put(-4,0){\line(1,-2){2}}

\put(0.1,-3.8){\makebox(0,0)[lb]{\(\langle 34\rangle\) (not coclass-settled)}}
\put(0.1,-5.8){\makebox(0,0)[lb]{\(\langle 174\rangle\)}}
\put(0.1,-7.8){\makebox(0,0)[lb]{\(\langle 978\rangle\)}}
\put(0.1,-9.8){\makebox(0,0)[lb]{\(\langle 6713\rangle\)}}
\put(0.1,-11.8){\makebox(0,0)[lb]{\(\langle 60885\rangle\)}}
\put(0.1,-13.8){\makebox(0,0)[lb]{\(1;1\)}}
\put(0.1,-15.8){\makebox(0,0)[lb]{\(1;1\)}}
\put(0.1,-17.8){\makebox(0,0)[lb]{\(1;1\)}}
\multiput(0,-4)(0,-2){8}{\circle*{0.1}}
\multiput(0,-4)(0,-2){7}{\line(0,-1){2}}
\put(0,-18){\vector(0,-1){2}}
\put(0,-20.2){\makebox(0,0)[ct]{\(\mathcal{T}_\ast^3(\langle 32,34\rangle)\)}}

\put(-1,-4){\makebox(0,0)[cc]{\(\mathcal{G}(2,3)\)}}

\put(-2.1,-4.2){\makebox(0,0)[rt]{\(\langle 35\rangle\)}}
\put(-2.1,-6.2){\makebox(0,0)[rt]{\(\langle 181\rangle\)}}
\put(-2.1,-8.2){\makebox(0,0)[rt]{\(\langle 984\rangle\)}}
\put(-2.1,-10.2){\makebox(0,0)[rt]{\(\langle 6719\rangle\)}}
\put(-2.1,-12.2){\makebox(0,0)[rt]{\(\langle 60891\rangle\)}}
\put(-2.1,-14.2){\makebox(0,0)[rt]{\(1;1\)}}
\put(-2.1,-16.2){\makebox(0,0)[rt]{\(1;1\)}}
\put(-2.1,-18.2){\makebox(0,0)[rt]{\(1;1\)}}
\multiput(-2,-4)(0,-2){8}{\circle*{0.1}}
\multiput(-2,-4)(0,-2){7}{\line(0,-1){2}}
\put(-2,-18){\vector(0,-1){2}}
\put(-2,-20.5){\makebox(0,0)[ct]{\(\mathcal{T}_\ast^3(\langle 32,35\rangle)\)}}

\put(-1,-6.2){\makebox(0,0)[ct]{\(\langle 175\rangle\)}}
\put(-1,-8.2){\makebox(0,0)[ct]{\(\langle 979\rangle\)}}
\put(-1,-10.2){\makebox(0,0)[ct]{\(\langle 6714\rangle\)}}
\put(-1,-12.2){\makebox(0,0)[ct]{\(\langle 60886\rangle\)}}
\put(-1,-14.2){\makebox(0,0)[ct]{\(1;2\)}}
\put(-1,-16.2){\makebox(0,0)[ct]{\(1;2\)}}
\put(-1,-18.2){\makebox(0,0)[ct]{\(1;2\)}}
\multiput(0,-4)(0,-2){7}{\line(-1,-2){1}}
\multiput(-1,-6)(0,-2){7}{\circle*{0.1}}

\put(0,-4){\line(1,-1){4}}
\put(0,-4){\line(1,-2){2}}
\put(1.1,-4.8){\makebox(0,0)[lb]{\(1^{\text{st}}\) bifurcation}}

\put(4.1,-7.8){\makebox(0,0)[lb]{\(\langle 444\rangle\) (not coclass-settled)}}
\put(4.1,-9.8){\makebox(0,0)[lb]{\(\langle 5503\rangle\)}}
\put(4.1,-11.8){\makebox(0,0)[lb]{\(\langle 58920\rangle\)}}
\put(4.1,-13.8){\makebox(0,0)[lb]{\(1;1\)}}
\put(4.1,-15.8){\makebox(0,0)[lb]{\(1;1\)}}
\put(4.1,-17.8){\makebox(0,0)[lb]{\(1;1\)}}
\multiput(4,-8)(0,-2){6}{\circle*{0.1}}
\multiput(4,-8)(0,-2){5}{\line(0,-1){2}}
\put(4,-18){\vector(0,-1){2}}
\put(4,-20.2){\makebox(0,0)[ct]{\(\mathcal{T}_\ast^4(\langle 128,444\rangle)\)}}

\put(3,-8){\makebox(0,0)[cc]{\(\mathcal{G}(2,4)\)}}

\put(1.9,-8.2){\makebox(0,0)[rt]{\(\langle 445\rangle\)}}
\put(1.9,-10.2){\makebox(0,0)[rt]{\(\langle 5509\rangle\)}}
\put(1.9,-12.2){\makebox(0,0)[rt]{\(\langle 58926\rangle\)}}
\put(1.9,-14.2){\makebox(0,0)[rt]{\(1;1\)}}
\put(1.9,-16.2){\makebox(0,0)[rt]{\(1;1\)}}
\put(1.9,-18.2){\makebox(0,0)[rt]{\(1;1\)}}
\multiput(2,-8)(0,-2){6}{\circle*{0.1}}
\multiput(2,-8)(0,-2){5}{\line(0,-1){2}}
\put(2,-18){\vector(0,-1){2}}
\put(2,-20.5){\makebox(0,0)[ct]{\(\mathcal{T}_\ast^4(\langle 128,445\rangle)\)}}

\put(3,-10.2){\makebox(0,0)[ct]{\(\langle 5504\rangle\)}}
\put(3,-12.2){\makebox(0,0)[ct]{\(\langle 58921\rangle\)}}
\put(3,-14.2){\makebox(0,0)[ct]{\(1;2\)}}
\put(3,-16.2){\makebox(0,0)[ct]{\(1;2\)}}
\put(3,-18.2){\makebox(0,0)[ct]{\(1;2\)}}
\multiput(4,-8)(0,-2){5}{\line(-1,-2){1}}
\multiput(3,-10)(0,-2){5}{\circle*{0.1}}

\put(4,-8){\line(1,-1){4}}
\put(4,-8){\line(1,-2){2}}
\put(5.1,-8.8){\makebox(0,0)[lb]{\(2^{\text{nd}}\) bifurcation}}

\put(8.1,-11.8){\makebox(0,0)[lb]{\(\langle 30599\rangle\) (not coclass-settled)}}
\put(8.1,-13.8){\makebox(0,0)[lb]{\(1;1\)}}
\put(8.1,-15.8){\makebox(0,0)[lb]{\(1;1\)}}
\put(8.1,-17.8){\makebox(0,0)[lb]{\(1;1\)}}
\multiput(8,-12)(0,-2){4}{\circle*{0.1}}
\multiput(8,-12)(0,-2){3}{\line(0,-1){2}}
\put(8,-18){\vector(0,-1){2}}
\put(8,-20.2){\makebox(0,0)[ct]{\(\mathcal{T}_\ast^5(\langle 512,30599\rangle)\)}}

\put(7,-12){\makebox(0,0)[cc]{\(\mathcal{G}(2,5)\)}}

\put(5.9,-12.2){\makebox(0,0)[rt]{\(\langle 30600\rangle\)}}
\put(5.9,-14.2){\makebox(0,0)[rt]{\(1;1\)}}
\put(5.9,-16.2){\makebox(0,0)[rt]{\(1;1\)}}
\put(5.9,-18.2){\makebox(0,0)[rt]{\(1;1\)}}
\multiput(6,-12)(0,-2){4}{\circle*{0.1}}
\multiput(6,-12)(0,-2){3}{\line(0,-1){2}}
\put(6,-18){\vector(0,-1){2}}
\put(6,-20.5){\makebox(0,0)[ct]{\(\mathcal{T}_\ast^5(\langle 512,30600\rangle)\)}}

\put(7,-14.2){\makebox(0,0)[ct]{\(1;2\)}}
\put(7,-16.2){\makebox(0,0)[ct]{\(1;2\)}}
\put(7,-18.2){\makebox(0,0)[ct]{\(1;2\)}}
\multiput(8,-12)(0,-2){3}{\line(-1,-2){1}}
\multiput(7,-14)(0,-2){3}{\circle*{0.1}}

\put(8,-12){\line(1,-1){4}}
\put(8,-12){\line(1,-2){2}}
\put(9.1,-12.8){\makebox(0,0)[lb]{\(3^{\text{rd}}\) bifurcation}}

\put(12.1,-15.8){\makebox(0,0)[lb]{\(2;1\) (not coclass-settled)}}
\put(12.1,-17.8){\makebox(0,0)[lb]{\(1;1\)}}
\multiput(12,-16)(0,-2){2}{\circle*{0.1}}
\multiput(12,-16)(0,-2){1}{\line(0,-1){2}}
\put(12,-18){\vector(0,-1){2}}
\put(12,-20.2){\makebox(0,0)[ct]{\(\mathcal{T}_\ast^6(\langle 512,30599\rangle-\#2;1)\)}}

\put(11,-16){\makebox(0,0)[cc]{\(\mathcal{G}(2,6)\)}}

\put(9.9,-16.2){\makebox(0,0)[rt]{\(2;2\)}}
\put(9.9,-18.2){\makebox(0,0)[rt]{\(1;1\)}}
\multiput(10,-16)(0,-2){2}{\circle*{0.1}}
\multiput(10,-16)(0,-2){1}{\line(0,-1){2}}
\put(10,-18){\vector(0,-1){2}}
\put(10,-20.5){\makebox(0,0)[ct]{\(\mathcal{T}_\ast^6(\langle 512,30599\rangle-\#2;2)\)}}

\put(11,-18.2){\makebox(0,0)[ct]{\(1;2\)}}
\multiput(12,-16)(0,-2){1}{\line(-1,-2){1}}
\multiput(11,-18)(0,-2){1}{\circle*{0.1}}

\put(12,-16){\line(1,-1){4}}
\put(12,-16){\line(1,-2){2}}
\put(13.1,-16.8){\makebox(0,0)[lb]{\(4^{\text{th}}\) bifurcation}}

\multiput(-2.5,-4)(0,-2){7}{\oval(1.5,1)}
\put(-3,-4.8){\makebox(0,0)[cc]{\underbar{\textbf{255}}}}
\put(-3,-6.8){\makebox(0,0)[cc]{\underbar{\textbf{1695}}}}
\put(-3,-8.8){\makebox(0,0)[cc]{\underbar{\textbf{3855}}}}
\put(-3,-10.8){\makebox(0,0)[cc]{\underbar{\textbf{12855}}}}
\put(-3,-12.8){\makebox(0,0)[cc]{\underbar{\textbf{124095}}}}
\put(-3,-14.8){\makebox(0,0)[cc]{\underbar{\textbf{331095}}}}
\put(-3,-16.8){\makebox(0,0)[cc]{\underbar{\textbf{1006095}}}}
\multiput(-1,-6)(0,-2){7}{\oval(1,1.5)}
\put(-1,-7){\makebox(0,0)[cc]{\underbar{\textbf{935}}}}
\put(-1,-9){\makebox(0,0)[cc]{\underbar{\textbf{1887}}}}
\put(-1,-11){\makebox(0,0)[cc]{\underbar{\textbf{6919}}}}
\put(-1,-13){\makebox(0,0)[cc]{\underbar{\textbf{88791}}}}
\put(-1,-15){\makebox(0,0)[cc]{\underbar{\textbf{86343}}}}
\put(-1,-17){\makebox(0,0)[cc]{\underbar{\textbf{256615}}}}
\put(-1,-19){\makebox(0,0)[cc]{\underbar{\textbf{746623}}}}

\multiput(1.5,-8)(0,-2){6}{\oval(1.5,1)}
\put(1,-8.8){\makebox(0,0)[cc]{\underbar{\textbf{1599}}}}
\put(1,-10.8){\makebox(0,0)[cc]{\underbar{\textbf{13767}}}}
\put(1,-12.8){\makebox(0,0)[cc]{\underbar{\textbf{47135}}}}
\put(1,-14.8){\makebox(0,0)[cc]{\underbar{\textbf{246831}}}}
\put(1,-16.8){\makebox(0,0)[cc]{\underbar{\textbf{371319}}}}
\put(1,-18.8){\makebox(0,0)[cc]{\underbar{\textbf{855231}}}}
\multiput(3,-10)(0,-2){5}{\oval(1,1.5)}
\put(3,-11){\makebox(0,0)[cc]{\underbar{\textbf{10735}}}}
\put(3,-13){\makebox(0,0)[cc]{\underbar{\textbf{19311}}}}
\put(3,-15){\makebox(0,0)[cc]{\underbar{\textbf{79663}}}}
\put(3,-17){\makebox(0,0)[cc]{\underbar{\textbf{103279}}}}
\put(3,-19){\makebox(0,0)[cc]{\underbar{\textbf{557887}}}}

\multiput(5.5,-12)(0,-2){3}{\oval(1.5,1)}
\put(5,-12.8){\makebox(0,0)[cc]{\underbar{\textbf{24415}}}}
\put(5,-14.8){\makebox(0,0)[cc]{\underbar{\textbf{63159}}}}
\put(5,-16.8){\makebox(0,0)[cc]{\underbar{\textbf{702519}}}}
\multiput(7,-14)(0,-2){3}{\oval(1,1.5)}
\put(7,-15){\makebox(0,0)[cc]{\underbar{\textbf{166463}}}}
\put(7,-17){\makebox(0,0)[cc]{\underbar{\textbf{395007}}}}
\put(7,-19){\makebox(0,0)[cc]{\underbar{\textbf{1116151}}}}

\multiput(9.5,-16)(0,-2){1}{\oval(1.5,1)}
\put(9,-16.8){\makebox(0,0)[cc]{\underbar{\textbf{231583}}}}
\multiput(11,-18)(0,-2){1}{\oval(1,1.5)}
\put(11,-19){\makebox(0,0)[cc]{\underbar{\textbf{1066407}}}}

\end{picture}

\end{figure}

}



\section{Computational results for two-stage towers}
\label{s:CompRslt}

With the aid of the computational algebra system MAGMA
\cite{MAGMA},
we have determined the pairs of parameters \((m,n)=(m(d),n(d))\),
investigated in
\cite{AZT},
for all \(11\,753\) square free radicands \(d=p_1p_2q\)
of the shape in Theorem
\ref{thm:TwoStage}
which occur in the range \(0<d<2\cdot 10^6\).
As mentioned at the beginning of \S\
\ref{s:TwoStage},
the result supports the conjecture that the corresponding
\(2\)-tower groups \(G_{m(d),n(d)}\) cover the pruned tree
\(\mathcal{T}_\ast(\langle 8,5\rangle)\) without gaps.

Recall that a pair \((m,n)\) contains information on
the \(2\)-class numbers of complex quadratic fields.
So we have a reduction of hard problems for biquadratic fields
to easy questions about quadratic fields.

By means of the following invariants,
the statistical distribution \(d\mapsto (m(d),n(d))\) of parameter pairs
is visualized on the pruned descendant tree \(\mathcal{T}_\ast(\langle 8,5\rangle)\),
using the injective (and probably even bijective) mapping \((m,n)\mapsto G_{m,n}\).
For each fixed individual pair \((m,n)\),
we define its \textit{minimal radicand} \(M(m,n)\) as an absolute invariant:

\begin{equation}
\label{eqn:StatInv}
M(m,n) := \min\lbrace d>0\mid (m(d),n(d))=(m,n)\rbrace.
\end{equation}

The purely group theoretic pruned descendant tree
was constructed in
\cite[\S\ 21.1, pp.182--184]{Ma1},
and was shown in
\cite[\S\ 10.4.1, Fig.7, p.175]{Ma1},
with vertices labelled by the standard identifiers in the SmallGroups Library
\cite{BEO1,BEO2}
or of the ANUPQ-package
\cite{GNO}.

In Figure
\ref{fig:IdAndParam222},
a pair \((m,n)\) of parameters is placed adjacent to the corresponding vertex \(G_{m,n}\)
of the pruned descendant tree \(\mathcal{T}_\ast(\langle 8,5\rangle)\).
There are no overlaps, since the mapping \((m,n)\mapsto G_{m,n}\) is injective.
Each vertex is additionally labelled with a formal identifier, as used in
\cite[Cor.21.1]{Ma1}.

In Figure
\ref{fig:MinRad222},
the minimal radicand \(M(m,n)\)
for which the adjacent vertex is realized as the corresponding group \(G_{m,n}\),
is shown underlined and with boldface font.

Vertices within the support of the distribution
are surrounded by an oval.
The oval is drawn in horizontal orientation for mainline vertices
and in vertical orientation for vertices in other periodic coclass sequences.



\section{Three-stage towers of \(3\)-class fields}
\label{s:ThreeStage}

Our second discovery of periodic bifurcations in trees of \(3\)-groups
will now be applied to a family of quadratic number fields \(K\)
with \(3\)-class group \(\mathrm{Cl}_3(K)\) of type \((3,3)\),
originally investigated by ourselves in
\cite{Ma,Ma0,Ma3},
and extended by Boston, Bush and Hajir in
\cite{BBH}.
The \(3\)-class tower groups \(G=\mathrm{G}_3^{\infty}(K)\) of these fields
are conjecturally distributed over six periodic sequences arising from repeated bifurcations
(of the new kind which was unknown up to now),
whereas it is proven that their metabelianizations populate
six well-known periodic coclass sequences of fixed coclass \(2\).

\begin{theorem}
\label{thm:ThreeStage}

Let \(K=\mathbb{Q}(\sqrt{d})\) be a complex quadratic field
with discriminant \(d<0\),
having a \(3\)-class group \(\mathrm{Cl}_3(K)\) of type \((3,3)\),
such that its \(3\)-principalization
in the four unramified cyclic cubic extensions \(L_1,\ldots,L_4\)
is given by one of the following two first layer TKTs
\[\varkappa_1(K)=(1,1,2,2)\text{ or }(3,1,2,2),\]
resp.
\[\varkappa_1(K)=(2,2,3,4)\text{ or }(2,3,3,4).\]
Further,
let the integer \(2\le\ell\le 9\) denote a foregiven upper bound.

\noindent
Then the \(3\)-class field tower of \(K\) is non-metabelian with exactly three stages,
and the isomorphism type of the Galois group
\(G=\mathrm{G}_3^{\infty}(K)=\mathrm{Gal}(\mathrm{F}_3^{\infty}(K)\vert K)\)
of the maximal unramified pro-\(3\) extension \(\mathrm{F}_3^{\infty}(K)\) of \(K\)
is characterized uniquely by the positive integer parameter \(2\le u\le\ell\)
defined by the \(3\)-class number
\(h_3(k_0)=3^u\)
of the simply real non-Galois cubic subfield
\(k_0\)
of the distinguished polarized extension \(L\) among \(L_1,\ldots,L_4\)
(i.e., \(L=L_1\), resp. \(L=L_2\)):

\begin{equation}
\label{eqn:BifurcationSequence}
\begin{aligned}
G &\simeq\langle 729,49\rangle(-\#2;1-\#1;1)^j-\#2;4\text{ or }5\vert 6,\text{ resp.}\\
G &\simeq\langle 729,54\rangle(-\#2;1-\#1;1)^j-\#2;2\text{ or }4\vert 6,\text{ with }j=u-2.
\end{aligned}
\end{equation}

\noindent
The metabelianization \(G/G^{\prime\prime}\) of the Schur \(\sigma\)-group \(G\), that is the Galois group
\(\mathrm{G}_3^2(K)=\mathrm{Gal}(\mathrm{F}_3^2(K)\vert K)\)
of the maximal metabelian unramified \(3\)-extension \(\mathrm{F}_3^2(K)\) of \(K\)
is unbalanced and given by

\begin{equation}
\label{eqn:CoclassSequence}
\begin{aligned}
G/G^{\prime\prime} &\simeq\langle 729,49\rangle(-\#1;1-\#1;1)^k-\#1;4\text{ or }5\vert 6,\text{ resp.}\\
G/G^{\prime\prime} &\simeq\langle 729,54\rangle(-\#1;1-\#1;1)^k-\#1;2\text{ or }4\vert 6,\text{ with }k=u-2.
\end{aligned}
\end{equation}

\end{theorem}

Again, we first state a corollary
whose proof will initialize the proof of Theorem
\ref{thm:ThreeStage}.



\begin{corollary}
\label{cor:ThreeStage}
Under the assumptions of Theorem
\ref{thm:ThreeStage},
the Artin pattern \(\mathrm{AP}(G)=(\tau(G);\varkappa(G))\)
of the \(3\)-tower group \(G=\mathrm{G}_3^\infty(K)\)
of the complex quadratic field \(K=\mathbb{Q}(\sqrt{d})\) is given as follows:

The ordered multi-layered transfer target type (TTT)
\(\tau(G)=\lbrack\tau_0;\tau_1;\tau_2\rbrack\)
of the Galois group \(G\) is given by \(\tau_0=(1^3)\), \(\tau_2=(u,u,1)\), and

\begin{equation}
\label{eqn:TTT}
\tau_1=
\begin{cases}
\lbrack (u+1,u),1^3,(2,1)^2\rbrack,\text{ if }G\in\mathcal{T}(\langle 729,49\rangle), \\
\lbrack (2,1),(u+1,u),(2,1)^2\rbrack,\text{ if }G\in\mathcal{T}(\langle 729,54\rangle).
\end{cases}
\end{equation}

\noindent
If we now denote by \(N_i:=\mathrm{Norm}_{L_i\vert K}(\mathrm{Cl}_3(L_i))\), \(1\le i\le 4\),
the norm class groups of the four unramified cyclic cubic extensions \(L_i\vert K\),
then the ordered multi-layered transfer kernel type (TKT)
\(\varkappa(G)=\lbrack\varkappa_0;\varkappa_1;\varkappa_2\rbrack\)
of the Galois group \(G\) is given by \(\varkappa_0=1\), \(\varkappa_2=(0)\), and

\begin{equation}
\label{eqn:TKT}
\varkappa_1=
\begin{cases}
(N_1,N_1,N_2,N_2)\text{ or }(N_3,N_1,N_2,N_2),\text{ if }G\in\mathcal{T}(\langle 729,49\rangle), \\
(N_2,N_2,N_3,N_4)\text{ or }(N_2,N_3,N_3,N_4),\text{ if }G\in\mathcal{T}(\langle 729,54\rangle).
\end{cases}
\end{equation}

\noindent
Thus, \(\varkappa_1\) is not a permutation of the norm class groups \(N_i\).
For \(G\in\mathcal{T}(\langle 729,49\rangle)\) it contains a single or no fixed point and no \(2\)-cycle,
and for \(G\in\mathcal{T}(\langle 729,54\rangle)\) it contains three or two fixed points and no \(2\)-cycle.

\end{corollary}



\begin{proof}
First, we must establish the connection of the TTT of \(G\) with the
distinguished non-Galois simply real cubic field \(k_0\).
Anticipating the partial result of Theorem
\ref{thm:ThreeStage}
that the metabelianization \(G/G^{\prime\prime}\) of \(G\) must be of coclass \(r=2\),
we can determine the \(3\)-class numbers of all four non-Galois cubic subfields \(k_i<L_i\)
with the aid of Theorem 4.2 in
\cite[p.489]{Ma0}:
with respect to the normalization in this theorem, we have
\(h_3(k_0)=3^u=h_3(k_1)=3^{\frac{m-2}{2}}\) and uniformly \(h_3(k_i)=3\) for \(2\le i\le 4\),
since \(e=r+1=3\), which implies \(\frac{e-1}{2}=1\),
and \(G/G^{\prime\prime}\) has no defect of commutativity.
The parameter \(m\) is the index of nilpotency of \(G/G^{\prime\prime}\),
whence the nilpotency class is given by \(c=m-1\).

Now, the statements are an immediate consequence of \S\S\ 4.1--4.2
in our recent article
\cite{Ma2},
where the claims are reduced to theorems in our earlier papers:
\cite[Thm.1.3, p.405]{Ma},
and, more generally,
\cite[Thm.4.4, p.440 and Tbl.4.7, p.441]{Ma3}.
We must only take into consideration that the \(3\)-class group \(\mathrm{Cl}_3(L)\) of \(L\)
is nearly homocyclic with abelian type invariants \(A(3,c)\simeq (u+1,u)\),
since \(u=\frac{m-2}{2}\), and thus \(2u+1=m-1=c\).
\end{proof}



\begin{proof}
(Proof of Theorem
\ref{thm:ThreeStage})
First, we use the Artin pattern of \(G\), as given in Corollary
\ref{cor:ThreeStage},
to narrow down the possibilities for \(G\).
The possible class-\(3\) quotients of \(G\)
are exactly the immediate descendants of
the common class-\(2\) quotient \(\langle 27,3\rangle\)
of all \(3\)-groups with abelianization of type \((3,3)\)
(apart from \(\langle 27,4\rangle\)), that is,
four vertices \(\langle 81,7\ldots 10\rangle\) of step size \(1\)
\cite[Fig.3]{Ma1},
and seven vertices \(\langle 243,3\ldots 9\rangle\) of step size \(2\)
\cite[Fig.4]{Ma1}.
All descendants of the former are of coclass \(1\) and
reveal the same three stable (i.e. parameter independent) components
\(((1^2)^3)\) of the first layer TTT \(\tau_1\),
according to
\cite[Thm.3.2,(1)]{Ma2},
which does not agree with the required TTT of \(G\).
Among the latter, the criterion
\cite[Cor.3.0.2, p.772]{BuMa}
rejects three of the seven vertices,
\(\langle 243,3\vert 4\vert 9\rangle\),
since the TKT of \(G\) does not contain a \(2\)-cycle,
and \(\langle 243,5\vert 7\rangle\) are discouraged, since they are terminal.
The remaining two vertices \(\langle 243,6\vert 8\rangle\)
are exactly the parents of the decisive groups \(\langle 729,49\vert 54\rangle\),
where periodic bifurcations set in.

Now, Theorem 21.3 and Corollaries 21.2--21.3 in
\cite[pp.185--187]{Ma1}
show that, using the local notation of Corollary 21.2,
\[G\simeq S_k:=\langle 729,49\vert 54\rangle(-\#2;1-\#1;1)^k-\#2;4\vert 5\vert 6\text{ resp. }2\vert 4\vert 6\]
and
\[G/G^{\prime\prime}\simeq V_{0,2k}:=\langle 729,49\vert 54\rangle(-\#1;1)^{2k}-\#1;4\vert 5\vert 6\text{ resp. }2\vert 4\vert 6,\]
both with \(k=u-2\).
\end{proof}



{\tiny

\begin{figure}[hb]
\caption{Minimal absolute discriminants \(\lvert d\rvert<10^8\) distributed over \(\mathcal{T}^2(\langle 243,6\rangle)\)}
\label{fig:MinDiscCocl2TreeQTyp33}


\setlength{\unitlength}{0.8cm}
\begin{picture}(17,22.5)(-7,-21.5)

\put(-8,0.5){\makebox(0,0)[cb]{order \(3^n\)}}
\put(-8,0){\line(0,-1){20}}
\multiput(-8.1,0)(0,-2){11}{\line(1,0){0.2}}
\put(-8.2,0){\makebox(0,0)[rc]{\(243\)}}
\put(-7.8,0){\makebox(0,0)[lc]{\(3^5\)}}
\put(-8.2,-2){\makebox(0,0)[rc]{\(729\)}}
\put(-7.8,-2){\makebox(0,0)[lc]{\(3^6\)}}
\put(-8.2,-4){\makebox(0,0)[rc]{\(2\,187\)}}
\put(-7.8,-4){\makebox(0,0)[lc]{\(3^7\)}}
\put(-8.2,-6){\makebox(0,0)[rc]{\(6\,561\)}}
\put(-7.8,-6){\makebox(0,0)[lc]{\(3^8\)}}
\put(-8.2,-8){\makebox(0,0)[rc]{\(19\,683\)}}
\put(-7.8,-8){\makebox(0,0)[lc]{\(3^9\)}}
\put(-8.2,-10){\makebox(0,0)[rc]{\(59\,049\)}}
\put(-7.8,-10){\makebox(0,0)[lc]{\(3^{10}\)}}
\put(-8.2,-12){\makebox(0,0)[rc]{\(177\,147\)}}
\put(-7.8,-12){\makebox(0,0)[lc]{\(3^{11}\)}}
\put(-8.2,-14){\makebox(0,0)[rc]{\(531\,441\)}}
\put(-7.8,-14){\makebox(0,0)[lc]{\(3^{12}\)}}
\put(-8.2,-16){\makebox(0,0)[rc]{\(1\,594\,323\)}}
\put(-7.8,-16){\makebox(0,0)[lc]{\(3^{13}\)}}
\put(-8.2,-18){\makebox(0,0)[rc]{\(4\,782\,969\)}}
\put(-7.8,-18){\makebox(0,0)[lc]{\(3^{14}\)}}
\put(-8.2,-20){\makebox(0,0)[rc]{\(14\,348\,907\)}}
\put(-7.8,-20){\makebox(0,0)[lc]{\(3^{15}\)}}
\put(-8,-20){\vector(0,-1){2}}

\put(-6,0.5){\makebox(0,0)[cb]{\(\tau_1(1)=\)}}
\put(-6,0){\makebox(0,0)[cc]{\((21)\)}}
\put(-6,-2){\makebox(0,0)[cc]{\((2^2)\)}}
\put(-6,-4){\makebox(0,0)[cc]{\((32)\)}}
\put(-6,-6){\makebox(0,0)[cc]{\((3^2)\)}}
\put(-6,-8){\makebox(0,0)[cc]{\((43)\)}}
\put(-6,-10){\makebox(0,0)[cc]{\((4^2)\)}}
\put(-6,-12){\makebox(0,0)[cc]{\((54)\)}}
\put(-6,-14){\makebox(0,0)[cc]{\((5^2)\)}}
\put(-6,-16){\makebox(0,0)[cc]{\((65)\)}}
\put(-6,-18){\makebox(0,0)[cc]{\((6^2)\)}}
\put(-6,-20){\makebox(0,0)[cc]{\((76)\)}}
\put(-6,-21){\makebox(0,0)[cc]{\textbf{TTT}}}
\put(-6.5,-21.2){\framebox(1,22){}}

\put(7.6,-7){\vector(0,1){3}}
\put(7.8,-7){\makebox(0,0)[lc]{depth \(3\)}}
\put(7.6,-7){\vector(0,-1){3}}

\put(-3.1,-8){\vector(0,1){2}}
\put(-3.3,-8){\makebox(0,0)[rc]{period length \(2\)}}
\put(-3.1,-8){\vector(0,-1){2}}

\put(0.7,-2){\makebox(0,0)[lc]{bifurcation from}}
\put(0.7,-2.3){\makebox(0,0)[lc]{\(\mathcal{G}(3,2)\) to \(\mathcal{G}(3,3)\)}}

\multiput(0,0)(0,-2){10}{\circle*{0.2}}
\multiput(0,0)(0,-2){9}{\line(0,-1){2}}
\multiput(-1,-2)(0,-2){10}{\circle*{0.2}}
\multiput(-2,-2)(0,-2){10}{\circle*{0.2}}
\multiput(1.95,-4.05)(0,-2){9}{\framebox(0.1,0.1){}}
\multiput(3,-2)(0,-2){10}{\circle*{0.2}}
\multiput(0,0)(0,-2){10}{\line(-1,-2){1}}
\multiput(0,0)(0,-2){10}{\line(-1,-1){2}}
\multiput(0,-2)(0,-2){9}{\line(1,-1){2}}
\multiput(0,0)(0,-2){10}{\line(3,-2){3}}
\multiput(-3.05,-4.05)(-1,0){2}{\framebox(0.1,0.1){}}
\multiput(3.95,-4.05)(0,-2){9}{\framebox(0.1,0.1){}}
\multiput(5,-4)(0,-2){9}{\circle*{0.1}}
\multiput(6,-4)(0,-2){9}{\circle*{0.1}}
\multiput(-1,-2)(-1,0){2}{\line(-1,-1){2}}
\multiput(3,-2)(0,-2){9}{\line(1,-2){1}}
\multiput(3,-2)(0,-2){9}{\line(1,-1){2}}
\multiput(3,-2)(0,-2){9}{\line(3,-2){3}}
\multiput(6.95,-6.05)(0,-2){8}{\framebox(0.1,0.1){}}
\multiput(6,-4)(0,-2){8}{\line(1,-2){1}}

\put(2,-0.5){\makebox(0,0)[lc]{branch}}
\put(2,-0.8){\makebox(0,0)[lc]{\(\mathcal{B}(5)\)}}
\put(2,-2.8){\makebox(0,0)[lc]{\(\mathcal{B}(6)\)}}
\put(2,-4.8){\makebox(0,0)[lc]{\(\mathcal{B}(7)\)}}
\put(2,-6.8){\makebox(0,0)[lc]{\(\mathcal{B}(8)\)}}
\put(2,-8.8){\makebox(0,0)[lc]{\(\mathcal{B}(9)\)}}
\put(2,-10.8){\makebox(0,0)[lc]{\(\mathcal{B}(10)\)}}
\put(2,-12.8){\makebox(0,0)[lc]{\(\mathcal{B}(11)\)}}
\put(2,-14.8){\makebox(0,0)[lc]{\(\mathcal{B}(12)\)}}
\put(2,-16.8){\makebox(0,0)[lc]{\(\mathcal{B}(13)\)}}
\put(2,-18.8){\makebox(0,0)[lc]{\(\mathcal{B}(14)\)}}

\put(-0.1,0.3){\makebox(0,0)[rc]{\(\langle 6\rangle\)}}
\put(-2.1,-1.8){\makebox(0,0)[rc]{\(\langle 50\rangle\)}}
\put(-1.1,-1.8){\makebox(0,0)[rc]{\(\langle 51\rangle\)}}
\put(0.1,-1.8){\makebox(0,0)[lc]{\(\langle 49\rangle\)}}
\put(3.1,-1.8){\makebox(0,0)[lc]{\(\langle 48\rangle\)}}
\put(-4.1,-3.5){\makebox(0,0)[cc]{\(\langle 292\rangle\)}}
\put(-3.1,-3.5){\makebox(0,0)[cc]{\(\langle 293\rangle\)}}
\put(-2.1,-3.3){\makebox(0,0)[cc]{\(\langle 289\rangle\)}}
\put(-2.1,-3.5){\makebox(0,0)[cc]{\(\langle 290\rangle\)}}
\put(-1.1,-3.5){\makebox(0,0)[cc]{\(\langle 288\rangle\)}}
\put(0.1,-3.5){\makebox(0,0)[lc]{\(\langle 285\rangle\)}}
\put(2.2,-3.3){\makebox(0,0)[cc]{\(\langle 284\rangle\)}}
\put(2.2,-3.5){\makebox(0,0)[cc]{\(\langle 291\rangle\)}}
\put(3.2,-3.3){\makebox(0,0)[cc]{\(\langle 286\rangle\)}}
\put(3.2,-3.5){\makebox(0,0)[cc]{\(\langle 287\rangle\)}}
\put(4.2,-3.3){\makebox(0,0)[cc]{\(\langle 276\rangle\)}}
\put(4.2,-3.5){\makebox(0,0)[cc]{\(\langle 283\rangle\)}}
\put(5.2,-3.1){\makebox(0,0)[cc]{\(\langle 280\rangle\)}}
\put(5.2,-3.3){\makebox(0,0)[cc]{\(\langle 281\rangle\)}}
\put(5.2,-3.5){\makebox(0,0)[cc]{\(\langle 282\rangle\)}}
\put(6.2,-3.1){\makebox(0,0)[cc]{\(\langle 277\rangle\)}}
\put(6.2,-3.3){\makebox(0,0)[cc]{\(\langle 278\rangle\)}}
\put(6.2,-3.5){\makebox(0,0)[cc]{\(\langle 279\rangle\)}}

\put(2.1,-3.8){\makebox(0,0)[lc]{\(*2\)}}
\multiput(-2.1,-3.8)(0,-4){5}{\makebox(0,0)[rc]{\(2*\)}}
\multiput(3.1,-3.8)(0,-4){5}{\makebox(0,0)[lc]{\(*2\)}}
\put(4.1,-3.8){\makebox(0,0)[lc]{\(*2\)}}
\multiput(5.1,-5.8)(0,-4){4}{\makebox(0,0)[lc]{\(*2\)}}
\multiput(5.5,-5.3)(0,-4){4}{\makebox(0,0)[lc]{\(\#2\)}}
\multiput(6.1,-5.8)(0,-4){4}{\makebox(0,0)[lc]{\(*2\)}}
\multiput(5.1,-3.8)(0,-4){5}{\makebox(0,0)[lc]{\(*3\)}}
\multiput(6.1,-3.8)(0,-4){5}{\makebox(0,0)[lc]{\(*3\)}}
\multiput(7.1,-5.8)(0,-2){4}{\makebox(0,0)[lc]{\(*3\)}}

\put(-3,-21){\makebox(0,0)[cc]{\textbf{TKT}}}
\put(-2,-21){\makebox(0,0)[cc]{E.14}}
\put(-1,-21){\makebox(0,0)[cc]{E.6}}
\put(0,-21){\makebox(0,0)[cc]{c.18}}
\put(2,-21){\makebox(0,0)[cc]{c.18}}
\put(3.1,-21){\makebox(0,0)[cc]{H.4}}
\put(4,-21){\makebox(0,0)[cc]{H.4}}
\put(5,-21){\makebox(0,0)[cc]{H.4}}
\put(6,-21){\makebox(0,0)[cc]{H.4}}
\put(7,-21){\makebox(0,0)[cc]{H.4}}
\put(-3,-21.5){\makebox(0,0)[cc]{\(\varkappa_1=\)}}
\put(-2,-21.5){\makebox(0,0)[cc]{\((3122)\)}}
\put(-1,-21.5){\makebox(0,0)[cc]{\((1122)\)}}
\put(0,-21.5){\makebox(0,0)[cc]{\((0122)\)}}
\put(2,-21.5){\makebox(0,0)[cc]{\((0122)\)}}
\put(3.1,-21.5){\makebox(0,0)[cc]{\((2122)\)}}
\put(4,-21.5){\makebox(0,0)[cc]{\((2122)\)}}
\put(5,-21.5){\makebox(0,0)[cc]{\((2122)\)}}
\put(6,-21.5){\makebox(0,0)[cc]{\((2122)\)}}
\put(7,-21.5){\makebox(0,0)[cc]{\((2122)\)}}
\put(-3.8,-21.7){\framebox(11.6,1){}}

\put(0,-18){\vector(0,-1){2}}
\put(0.2,-19.4){\makebox(0,0)[lc]{infinite}}
\put(0.2,-19.9){\makebox(0,0)[lc]{mainline}}
\put(1.8,-20.4){\makebox(0,0)[rc]{\(\mathcal{T}^2(\langle 243,6\rangle)\)}}

\multiput(-1,-4)(0,-4){4}{\oval(1,1)}
\put(-1,-4.8){\makebox(0,0)[lc]{\underbar{\textbf{15\,544}}}}
\put(-1,-8.8){\makebox(0,0)[lc]{\underbar{\textbf{268\,040}}}}
\put(-1,-12.8){\makebox(0,0)[lc]{\underbar{\textbf{1\,062\,708}}}}
\put(-1,-16.8){\makebox(0,0)[lc]{\underbar{\textbf{27\,629\,107}}}}
\multiput(-2,-4)(0,-4){4}{\oval(1,1)}
\put(-2,-4.8){\makebox(0,0)[rc]{\underbar{\textbf{16\,627}}}}
\put(-2,-8.8){\makebox(0,0)[rc]{\underbar{\textbf{262\,744}}}}
\put(-2,-12.8){\makebox(0,0)[rc]{\underbar{\textbf{4\,776\,071}}}}
\put(-2,-16.8){\makebox(0,0)[rc]{\underbar{\textbf{40\,059\,363}}}}
\multiput(6,-6)(0,-4){4}{\oval(1,1)}
\put(6,-6.8){\makebox(0,0)[rc]{\underbar{\textbf{21\,668}}}}
\put(6,-10.8){\makebox(0,0)[rc]{\underbar{\textbf{446\,788}}}}
\put(6,-14.8){\makebox(0,0)[rc]{\underbar{\textbf{3\,843\,907}}}}
\put(6,-18.8){\makebox(0,0)[rc]{\underbar{\textbf{52\,505\,588}}}}

\end{picture}

\end{figure}
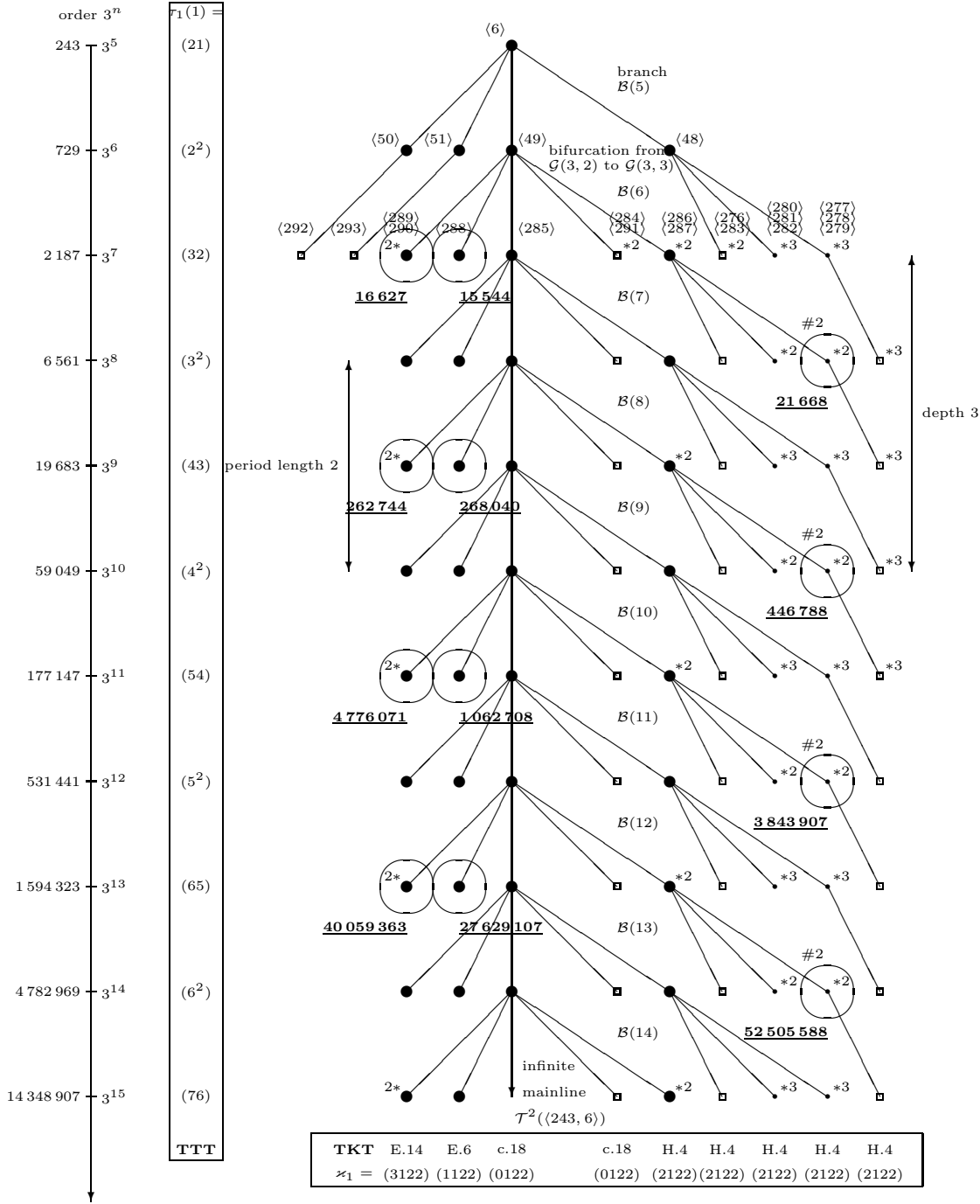

}



\section{Computational results for three-stage towers}
\label{s:CompRsltThreeSage}

With the aid of the computational algebra system MAGMA
\cite{MAGMA},
where the class field theoretic techniques of Fieker
\cite{Fi}
are implemented,
we have determined the Artin pattern \((\tau(K);\varkappa(K))\)
of all complex quadratic fields \(K=\mathbb{Q}(\sqrt{d})\)
with discriminants in the range \(-10^8<d<0\),
whose first layer TTT \(\tau_1(K)\) had been precomputed by Boston, Bush and Hajir
in the database underlying the numerical results in
\cite{BBH}.

Figure
\ref{fig:MinDiscCocl2TreeQTyp33},
resp.
\ref{fig:MinDiscCocl2TreeUTyp33},
shows the minimal absolute discriminant \(\lvert d\rvert\),
underlined and with boldface font,
for which the adjacent vertex of the coclass tree
\(\mathcal{T}^2(\langle 729,49\rangle)\),
resp.
\(\mathcal{T}^2(\langle 729,54\rangle)\),
is realized as the metabelianization \(G/G^{\prime\prime}\)
of the \(3\)-tower group \(G\) of \(K=\mathbb{Q}(\sqrt{d})\).
Vertices within the support of the distribution
are surrounded by an oval.
The corresponding projections \(G\to G/G^{\prime\prime}\)
have been visualized in the Figures 8--9 of
\cite[pp.188--189]{Ma1}.

We have published this information in the Online Encyclopedia of Integer Sequences (OEIS)
\cite{OEIS},
sequences A247692 to A247697.



{\tiny

\begin{figure}[hb]
\caption{Minimal absolute discriminants \(\lvert d\rvert<10^8\) distributed over \(\mathcal{T}^2(\langle 243,8\rangle)\)}
\label{fig:MinDiscCocl2TreeUTyp33}


\setlength{\unitlength}{0.8cm}
\begin{picture}(17,22.5)(-7,-21.5)

\put(-8,0.5){\makebox(0,0)[cb]{order \(3^n\)}}
\put(-8,0){\line(0,-1){20}}
\multiput(-8.1,0)(0,-2){11}{\line(1,0){0.2}}
\put(-8.2,0){\makebox(0,0)[rc]{\(243\)}}
\put(-7.8,0){\makebox(0,0)[lc]{\(3^5\)}}
\put(-8.2,-2){\makebox(0,0)[rc]{\(729\)}}
\put(-7.8,-2){\makebox(0,0)[lc]{\(3^6\)}}
\put(-8.2,-4){\makebox(0,0)[rc]{\(2\,187\)}}
\put(-7.8,-4){\makebox(0,0)[lc]{\(3^7\)}}
\put(-8.2,-6){\makebox(0,0)[rc]{\(6\,561\)}}
\put(-7.8,-6){\makebox(0,0)[lc]{\(3^8\)}}
\put(-8.2,-8){\makebox(0,0)[rc]{\(19\,683\)}}
\put(-7.8,-8){\makebox(0,0)[lc]{\(3^9\)}}
\put(-8.2,-10){\makebox(0,0)[rc]{\(59\,049\)}}
\put(-7.8,-10){\makebox(0,0)[lc]{\(3^{10}\)}}
\put(-8.2,-12){\makebox(0,0)[rc]{\(177\,147\)}}
\put(-7.8,-12){\makebox(0,0)[lc]{\(3^{11}\)}}
\put(-8.2,-14){\makebox(0,0)[rc]{\(531\,441\)}}
\put(-7.8,-14){\makebox(0,0)[lc]{\(3^{12}\)}}
\put(-8.2,-16){\makebox(0,0)[rc]{\(1\,594\,323\)}}
\put(-7.8,-16){\makebox(0,0)[lc]{\(3^{13}\)}}
\put(-8.2,-18){\makebox(0,0)[rc]{\(4\,782\,969\)}}
\put(-7.8,-18){\makebox(0,0)[lc]{\(3^{14}\)}}
\put(-8.2,-20){\makebox(0,0)[rc]{\(14\,348\,907\)}}
\put(-7.8,-20){\makebox(0,0)[lc]{\(3^{15}\)}}
\put(-8,-20){\vector(0,-1){2}}

\put(-6,0.5){\makebox(0,0)[cb]{\(\tau_1(2)=\)}}
\put(-6,0){\makebox(0,0)[cc]{\((21)\)}}
\put(-6,-2){\makebox(0,0)[cc]{\((2^2)\)}}
\put(-6,-4){\makebox(0,0)[cc]{\((32)\)}}
\put(-6,-6){\makebox(0,0)[cc]{\((3^2)\)}}
\put(-6,-8){\makebox(0,0)[cc]{\((43)\)}}
\put(-6,-10){\makebox(0,0)[cc]{\((4^2)\)}}
\put(-6,-12){\makebox(0,0)[cc]{\((54)\)}}
\put(-6,-14){\makebox(0,0)[cc]{\((5^2)\)}}
\put(-6,-16){\makebox(0,0)[cc]{\((65)\)}}
\put(-6,-18){\makebox(0,0)[cc]{\((6^2)\)}}
\put(-6,-20){\makebox(0,0)[cc]{\((76)\)}}
\put(-6,-21){\makebox(0,0)[cc]{\textbf{TTT}}}
\put(-6.5,-21.2){\framebox(1,22){}}

\put(7.6,-7){\vector(0,1){3}}
\put(7.8,-7){\makebox(0,0)[lc]{depth \(3\)}}
\put(7.6,-7){\vector(0,-1){3}}

\put(-3.1,-8){\vector(0,1){2}}
\put(-3.3,-8){\makebox(0,0)[rc]{period length \(2\)}}
\put(-3.1,-8){\vector(0,-1){2}}

\put(0.7,-2){\makebox(0,0)[lc]{bifurcation from}}
\put(0.7,-2.3){\makebox(0,0)[lc]{\(\mathcal{G}(3,2)\) to \(\mathcal{G}(3,3)\)}}

\multiput(0,0)(0,-2){10}{\circle*{0.2}}
\multiput(0,0)(0,-2){9}{\line(0,-1){2}}
\multiput(-1,-2)(0,-2){10}{\circle*{0.2}}
\multiput(-2,-2)(0,-2){10}{\circle*{0.2}}
\multiput(1.95,-4.05)(0,-2){9}{\framebox(0.1,0.1){}}
\multiput(3,-2)(0,-2){10}{\circle*{0.2}}
\multiput(0,0)(0,-2){10}{\line(-1,-2){1}}
\multiput(0,0)(0,-2){10}{\line(-1,-1){2}}
\multiput(0,-2)(0,-2){9}{\line(1,-1){2}}
\multiput(0,0)(0,-2){10}{\line(3,-2){3}}
\multiput(-3.05,-4.05)(-1,0){2}{\framebox(0.1,0.1){}}
\multiput(3.95,-6.05)(0,-2){8}{\framebox(0.1,0.1){}}
\multiput(5,-6)(0,-2){8}{\circle*{0.1}}
\multiput(6,-4)(0,-2){9}{\circle*{0.1}}
\multiput(-1,-2)(-1,0){2}{\line(-1,-1){2}}
\multiput(3,-4)(0,-2){8}{\line(1,-2){1}}
\multiput(3,-4)(0,-2){8}{\line(1,-1){2}}
\multiput(3,-2)(0,-2){9}{\line(3,-2){3}}
\multiput(6.95,-6.05)(0,-2){8}{\framebox(0.1,0.1){}}
\multiput(6,-4)(0,-2){8}{\line(1,-2){1}}

\put(2,-0.5){\makebox(0,0)[lc]{branch}}
\put(2,-0.8){\makebox(0,0)[lc]{\(\mathcal{B}(5)\)}}
\put(2,-2.8){\makebox(0,0)[lc]{\(\mathcal{B}(6)\)}}
\put(2,-4.8){\makebox(0,0)[lc]{\(\mathcal{B}(7)\)}}
\put(2,-6.8){\makebox(0,0)[lc]{\(\mathcal{B}(8)\)}}
\put(2,-8.8){\makebox(0,0)[lc]{\(\mathcal{B}(9)\)}}
\put(2,-10.8){\makebox(0,0)[lc]{\(\mathcal{B}(10)\)}}
\put(2,-12.8){\makebox(0,0)[lc]{\(\mathcal{B}(11)\)}}
\put(2,-14.8){\makebox(0,0)[lc]{\(\mathcal{B}(12)\)}}
\put(2,-16.8){\makebox(0,0)[lc]{\(\mathcal{B}(13)\)}}
\put(2,-18.8){\makebox(0,0)[lc]{\(\mathcal{B}(14)\)}}

\put(-0.1,0.3){\makebox(0,0)[rc]{\(\langle 8\rangle\)}}
\put(-2.1,-1.8){\makebox(0,0)[rc]{\(\langle 53\rangle\)}}
\put(-1.1,-1.8){\makebox(0,0)[rc]{\(\langle 55\rangle\)}}
\put(0.1,-1.8){\makebox(0,0)[lc]{\(\langle 54\rangle\)}}
\put(3.1,-1.8){\makebox(0,0)[lc]{\(\langle 52\rangle\)}}
\put(-4.1,-3.5){\makebox(0,0)[cc]{\(\langle 300\rangle\)}}
\put(-3.1,-3.5){\makebox(0,0)[cc]{\(\langle 309\rangle\)}}
\put(-2.1,-3.3){\makebox(0,0)[cc]{\(\langle 302\rangle\)}}
\put(-2.1,-3.5){\makebox(0,0)[cc]{\(\langle 306\rangle\)}}
\put(-1.1,-3.5){\makebox(0,0)[cc]{\(\langle 304\rangle\)}}
\put(0.1,-3.5){\makebox(0,0)[lc]{\(\langle 303\rangle\)}}
\put(2.2,-3.3){\makebox(0,0)[cc]{\(\langle 307\rangle\)}}
\put(2.2,-3.5){\makebox(0,0)[cc]{\(\langle 308\rangle\)}}
\put(3.2,-3.3){\makebox(0,0)[cc]{\(\langle 301\rangle\)}}
\put(3.2,-3.5){\makebox(0,0)[cc]{\(\langle 305\rangle\)}}
\put(6.2,-2.5){\makebox(0,0)[cc]{\(\langle 294\rangle\)}}
\put(6.2,-2.7){\makebox(0,0)[cc]{\(\langle 295\rangle\)}}
\put(6.2,-2.9){\makebox(0,0)[cc]{\(\langle 296\rangle\)}}
\put(6.2,-3.1){\makebox(0,0)[cc]{\(\langle 297\rangle\)}}
\put(6.2,-3.3){\makebox(0,0)[cc]{\(\langle 298\rangle\)}}
\put(6.2,-3.5){\makebox(0,0)[cc]{\(\langle 299\rangle\)}}

\put(2.1,-3.8){\makebox(0,0)[lc]{\(*2\)}}
\multiput(-2.1,-3.8)(0,-4){5}{\makebox(0,0)[rc]{\(2*\)}}
\multiput(3.1,-3.8)(0,-4){5}{\makebox(0,0)[lc]{\(*2\)}}
\put(6.1,-3.8){\makebox(0,0)[lc]{\(*6\)}}
\multiput(5.1,-5.8)(0,-4){4}{\makebox(0,0)[lc]{\(*2\)}}
\multiput(5.5,-5.3)(0,-4){4}{\makebox(0,0)[lc]{\(\#4\)}}
\multiput(6.1,-5.8)(0,-4){4}{\makebox(0,0)[lc]{\(*2\)}}
\multiput(5.1,-7.8)(0,-4){4}{\makebox(0,0)[lc]{\(*3\)}}
\multiput(6.1,-7.8)(0,-4){4}{\makebox(0,0)[lc]{\(*3\)}}
\multiput(7.1,-7.8)(0,-2){4}{\makebox(0,0)[lc]{\(*2\)}}

\put(-3,-21){\makebox(0,0)[cc]{\textbf{TKT}}}
\put(-2,-21){\makebox(0,0)[cc]{E.9}}
\put(-1,-21){\makebox(0,0)[cc]{E.8}}
\put(0,-21){\makebox(0,0)[cc]{c.21}}
\put(2,-21){\makebox(0,0)[cc]{c.21}}
\put(3.1,-21){\makebox(0,0)[cc]{G.16}}
\put(4,-21){\makebox(0,0)[cc]{G.16}}
\put(5,-21){\makebox(0,0)[cc]{G.16}}
\put(6,-21){\makebox(0,0)[cc]{G.16}}
\put(7,-21){\makebox(0,0)[cc]{G.16}}
\put(-3,-21.5){\makebox(0,0)[cc]{\(\varkappa_1=\)}}
\put(-2,-21.5){\makebox(0,0)[cc]{\((2334)\)}}
\put(-1,-21.5){\makebox(0,0)[cc]{\((2234)\)}}
\put(0,-21.5){\makebox(0,0)[cc]{\((2034)\)}}
\put(2,-21.5){\makebox(0,0)[cc]{\((2034)\)}}
\put(3.1,-21.5){\makebox(0,0)[cc]{\((2134)\)}}
\put(4,-21.5){\makebox(0,0)[cc]{\((2134)\)}}
\put(5,-21.5){\makebox(0,0)[cc]{\((2134)\)}}
\put(6,-21.5){\makebox(0,0)[cc]{\((2134)\)}}
\put(7,-21.5){\makebox(0,0)[cc]{\((2134)\)}}
\put(-3.8,-21.7){\framebox(11.6,1){}}

\put(0,-18){\vector(0,-1){2}}
\put(0.2,-19.4){\makebox(0,0)[lc]{infinite}}
\put(0.2,-19.9){\makebox(0,0)[lc]{mainline}}
\put(1.8,-20.4){\makebox(0,0)[rc]{\(\mathcal{T}^2(\langle 243,8\rangle)\)}}

\multiput(-1,-4)(0,-4){4}{\oval(1,1)}
\put(-1,-4.8){\makebox(0,0)[lc]{\underbar{\textbf{34\,867}}}}
\put(-1,-8.8){\makebox(0,0)[lc]{\underbar{\textbf{370\,740}}}}
\put(-1,-12.8){\makebox(0,0)[lc]{\underbar{\textbf{4\,087\,295}}}}
\put(-1,-16.8){\makebox(0,0)[lc]{\underbar{\textbf{19\,027\,947}}}}
\multiput(-2,-4)(0,-4){5}{\oval(1,1)}
\put(-2,-4.8){\makebox(0,0)[rc]{\underbar{\textbf{9\,748}}}}
\put(-2,-8.8){\makebox(0,0)[rc]{\underbar{\textbf{297\,079}}}}
\put(-2,-12.8){\makebox(0,0)[rc]{\underbar{\textbf{1\,088\,808}}}}
\put(-2,-16.8){\makebox(0,0)[rc]{\underbar{\textbf{11\,091\,140}}}}
\put(-2,-19.3){\makebox(0,0)[rc]{\underbar{\textbf{94\,880\,548}}}}
\multiput(6,-6)(0,-4){4}{\oval(1,1)}
\put(6,-6.8){\makebox(0,0)[rc]{\underbar{\textbf{17\,131}}}}
\put(6,-10.8){\makebox(0,0)[rc]{\underbar{\textbf{819\,743}}}}
\put(6,-14.8){\makebox(0,0)[rc]{\underbar{\textbf{2\,244\,399}}}}
\put(6,-18.8){\makebox(0,0)[rc]{\underbar{\textbf{30\,224\,744}}}}

\end{picture}

\end{figure}
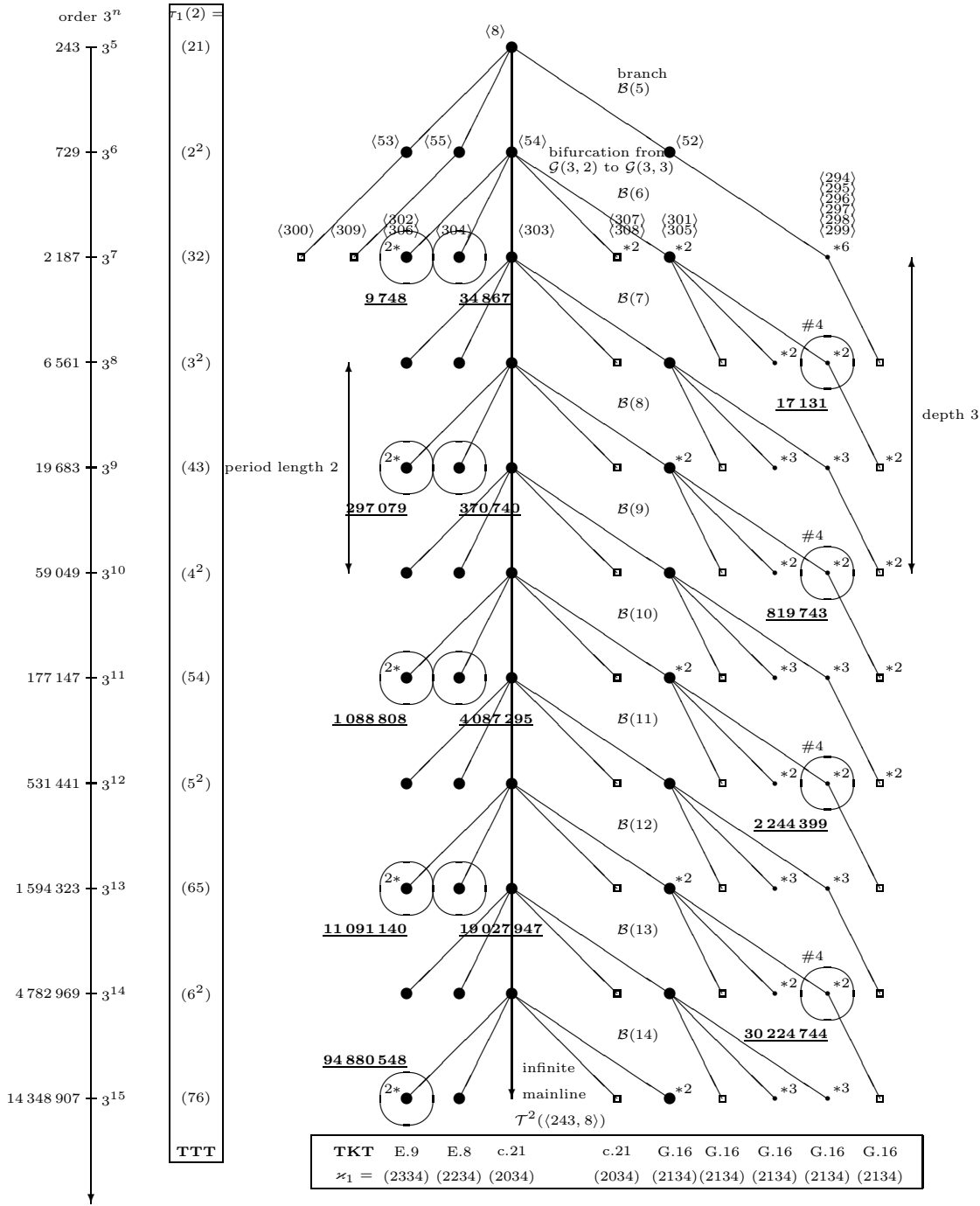

}



We emphasize that the results of section
\ref{s:ThreeStage}
provide the background for considerably stronger assertions
than those made in
\cite{BuMa}.
Firstly, since they concern four TKTs E.6, E.14, E.8, E.9 instead of just TKT E.9
\cite[\S\ 4]{Ma2},
and secondly, since they apply to varying odd nilpotency class \(5\le\mathrm{cl}(G)\le 19\)
instead of just class \(5\).



\section{Acknowledgements}
\label{s:Acknowledgements}

We gratefully acknowledge that our research is supported by the
Austrian Science Fund (FWF): P 26008-N25.
We are indebted to Nigel Boston, Michael R. Bush and Farshid Hajir
for kindly making available an unpublished database containing
numerical results of their paper
\cite{BBH}.




\end{document}